\documentclass{article}
\usepackage[lmargin=2cm, rmargin=2cm, tmargin=2cm, bmargin=2cm]{geometry}
\usepackage{tikz-cd}
\usepackage{amsmath, amsthm, amssymb}
\usepackage{hyperref}
\usepackage{cite}
\usepackage{enumerate}
\usepackage{thmtools}
\usepackage{thm-restate}

\newcommand{\ts}[2]{{_{#1}}\!\times_{#2}\!}
\title{Vector bundles and differential bundles in the category of smooth manifolds}
\author{
    Benjamin MacAdam
    \footnote{Department of Computer Science, University of Calgary, Calgary, Canada.}
    \footnote{A very special thanks to Matthew Burke, who provided crucial insights to earlier versions of this work.}
    \footnote{Benjamin MacAdam declares they have no competing financial interests.}
    }

\usepackage[capitalize]{cleveref}

\newtheorem{theorem}{Theorem}[subsection]
\newtheorem{corollary}[theorem]{Corollary}
\newtheorem{proposition}[theorem]{Proposition}
\theoremstyle{definition}
\newtheorem{definition}[theorem]{Definition}

\newtheorem{lemma}[theorem]{Lemma}
\theoremstyle{remark}
\newtheorem{remark}[theorem]{Remark}
\newtheorem{example}[theorem]{Example}

\begin{document}
\maketitle
\abstract{
  A tangent category is a category equipped with an endofunctor that satisfies certain axioms which capture the abstract properties of the tangent bundle functor from classical differential geometry.
  Cockett and Cruttwell introduced differential bundles in 2017 as an algebraic alternative to vector bundles in an arbitrary tangent category. 
  In this paper, we prove that differential bundles in the category of smooth manifolds are precisely vector bundles. 
  In particular, this means that we can give a characterisation of vector bundles that exhibits them as models of a tangent categorical essentially algebraic theory.
}

\tableofcontents

\section{Introduction}
\label{sec:introduction}

A smooth vector bundle is a model of an algebraic theory in the category of smooth manifolds that satisfies an additional topological axiom.
If $q: E \rightarrow M$ is the underlying projection of the vector bundle then these axioms include the data of a zero section $\xi: M \rightarrow E$, an addition $+_q: E\ts{q}{q} E \rightarrow E$ and a scalar multiplication $\bullet_q:\mathbb{R}\times E \rightarrow E$ that satisfy the appropriate axioms describing an $\mathbb{R}$-module in the slice category over $M$.
The additional topological axiom is that vector bundles are \emph{locally trivial}.
This means that $q: E \rightarrow M$ is locally isomorphic to a projection $\pi_0: U\times \mathbb{R}^n \rightarrow U$ for some open set $U$ and natural number $n$.
The advantage of using the local triviality condition as part of the definition of a smooth vector bundle is that it makes clear how to perform calculations using local coordinates.
However, the local triviality condition axiomatises the existence of a trivialization, which is not an algebraic condition.
%Therefore, the theory of smooth vector bundles is not sketchable in the classical sense - in this paper, we show that it is sketchable in the tangent categorical sense.

The main results of this paper are about \emph{differential bundles} in a tangent category.
A tangent category consists of a category $\mathbb{X}$ equipped with an endofunctor $T$ on $\mathbb{X}$ that satisfies axioms which capture the abstract properties of the tangent bundle functor from classical differential geometry.
The idea behind the definition of a differential bundle is to axiomatise the following fundamental property of vector bundles: if $q: E \rightarrow M$ is a vector bundle, and $x\in M$ and $v\in E_x$ is a vector in the fibre above $x$, then the tangent space $T_v(E_x)$ is naturally identified with $E_x$.
In particular every differential bundle $q:E \rightarrow M$ has a \emph{universal lift} $\mu:E\ts{q}{q}E \rightarrow T(E)$ which is analogous to the map $\mu(a, b) = \frac{d}{dt}|_0\left(a+t\bullet_q b\right)$\footnote{
    	In this paper, we make use of applicative composition of morphisms $f \circ g = A \xrightarrow{g} B \xrightarrow{f} C$
    } in the theory of vector bundles and a \emph{vertical lift} $\lambda:E \rightarrow T(E)$ given by $\lambda(e) = \mu(0, e)$.
(In Section 1 of \cite{MR1618130} these maps are called the \emph{big} and \emph{small vertical lifts} respectively and the big lift also appears in 6.11 of \cite{MR1202431}.)
Then the local triviality condition is replaced by the algebraic condition (in the sense of Freyd and Kelly \cite{MR322004}) that
\[\begin{tikzcd}
    E\ts{q}{q} E \rar{\mu} \dar{q\pi_0} & T(E) \dar{T(q)}\\
    M \rar{0} & T(M)
\end{tikzcd}\]
is a pullback preserved by iterated applications of the tangent bundle functor $T$.
%The precise definition of a differential bundle is in \cref{sub:alternative_differential_bundle}.
The work in \cite{MR3792842} and \cite{MR3684725} has provided evidence that differential bundles are an appropriate generalisation of vector bundles and many of the results of classical differential geometry concerning vector bundles hold for differential bundles.
However, there is no proof of the equivalence between the vector bundles and differential bundles in the category of smooth manifolds in the literature.
In this paper, we give a proof of this result.
Specifically in \cref{sec:differential_bundles_are_vector_bundles} we prove:
\begin{restatable*}{theorem}{isoOfCategories}\label{thm:iso-of-categories}
    The category of differential bundles (with linear or bundle morphisms) in the category of smooth manifolds is isomorphic to the category of smooth vector bundles (with linear or bundle morphisms, respectively).
\end{restatable*}
Our proof makes use of a more general result, which proves that in a tangent category with negatives every differential bundle is the retract of a pullback of the tangent projection on its total space (\cref{cor:splitting-of-differential-bundles}).
The category of vector bundles is closed to idempotent splittings and reindexing, so the result follows.

%Our proof of \cref{thm:iso-of-categories} uses the more general result \cref{cor:splitting-of-differential-bundles} which proves that in a tangent category with negatives every differential bundle is the retract of a pullback of the tangent projection on its total space.
%In \cref{sec:dbun-are-vector} we deduce \cref{thm:iso-of-categories} from \cref{cor:splitting-of-differential-bundles} using the fact that vector bundles are closed under pullback and idempotent splittings.

In addition to making rigorous the relationship between differential bundles and vector bundles, we describe some alternative characterisations of differential bundles.
To do this, we introduce pre-differential bundles, which captures the equational fragment of the definition of a differential bundle.
In \cref{sec:differential_bundles} we show how a universal property on a pre-differential bundle induces an addition map.
Later we give the following characterisation of differential bundles in the category of smooth manifolds:
\begin{restatable*}{corollary}{essentialAlgebraicDescent}
  \label{intro:essentially-algebraic-definition}
%  If $q:E \rightarrow M$, $\xi:M \rightarrow E$ and $\lambda: E \rightarrow T(E)$ are smooth functions then $q$ is a differential bundle with zero section $\xi$ and (small) vertical lift $\lambda$ iff $q\xi = id$, $\ell\lambda = T(\lambda)\lambda$, $\lambda\xi = 0\xi$ and
  A pre-differential bundle $(q: E \to M, \xi, \lambda)$ in the category of smooth manifolds is a differential bundle if and only if
  \begin{equation*}
    \begin{tikzcd}[ampersand replacement=\&]
      E \rar{\lambda} \dar{q} \& TE \dar{(p,Tq)} \\
      M \rar{(\xi,0)} \& E \times TM
    \end{tikzcd} 
  \end{equation*}
  is a pullback.
\end{restatable*}

We prove this using \textit{retractive} display systems and \textit{strong} differential bundles.
A differential bundle $E \xrightarrow{q} M$ satisfies the universal property that its vertical bundle - the kernel of $Tq: TE \to TM$ - decomposes into the fibre product $E \ts{q}{q} E$.
A differential bundle $E \xrightarrow{q} M$ is \textit{strong} when the kernel of the projection $p: TE \to E$ decomposes into a fibre product of $E$ and $TM$ over $M$, the precise definition of a strong differential bundle is in \cref{def:strong-differential-bundle}.
A \emph{proper retractive display system} is a class of maps in a tangent category that satisfies certain axioms which capture the abstract properties of the class of submersions in the category of smooth manifolds.
(A smooth function is a submersion whose derivative at every point is surjective.)
The theory of retractive display systems and the main example (tangent categories where the \textit{tangent submersions} form a display system) is developed in \cref{sub:retractive-display}.

\subsection{Tangent categories}
\label{subsec:tangent_categories}

In this section, we recall the definition of a tangent category. A tangent
category consists of a category $\mathbb{X}$ equipped with a 
structure that axiomatises the properties of the tangent bundle functor from
classical differential geometry. The idea of a tangent category originated in
\cite{MR2725166} and was further developed in \cite{MR3192082}. We begin
with the definition of an additive bundle which is a basic building block for
the theory.

\begin{definition}
  An \emph{additive bundle over $M$} is a commutative monoid in the slice
  category $\mathbb{X}/M$. Explicitly: an additive bundle consists of a
  projection $q:E \rightarrow M$, an addition $+:E\ts{q}{q}E \rightarrow E$ and
  a zero $\xi:M \rightarrow E$ satisfying the usual axioms for a commutative monoid.
\end{definition}

We often use $E_n$ to denote the $n$-fold pullback
$E\ts{q}{q}E\ts{q}{q}...\ts{q}{q}E$, and write the addition map as an infix operation
$\pi_0 +_q \pi_1$. 

\begin{definition}
  If $q$ and $q'$ are additive bundles then \emph{an additive bundle morphism $f:q \Rightarrow q'$} is a square
  \[\begin{tikzcd}
      E \dar{q} \rar{f_1} & E' \dar{q'}\\
      M \rar{f_0} & M'
    \end{tikzcd}\] such that $(f_1\pi_0 +_{q'} f_1\pi_1) = f_1 (\pi_0+_q \pi_1)$
     and $f_1\xi = \xi'f_0$.
\end{definition}

The following definition describes a specific type of limit that repeatedly occurs in the theory of tangent categories.

\begin{definition}
  If $T$ is an endofunctor on $\mathbb{X}$ then a limit diagram in $\mathbb{X}$ is a \emph{$T$-limit diagram} if and only if it is preserved by $T^n$ for all $n\in \mathbb{N}$ (we will often simply say ``$T$-limit''). We write $T$-pullback, $T$-equaliser etc. for the appropriate specialisations of this definition.
\end{definition}

The following is definition 2.1 in \cite{MR3792842}.

\begin{definition}
  A \emph{tangent category} is a category $\mathbb{X}$ equipped with:
  \begin{itemize}
  \item an endofunctor $T$ on $\mathbb{X}$
  \item a natural transformation $p:T \Rightarrow id_{\mathbb{X}}$ such that
    each $n$-fold pullback $T_n:=T\ts{p}{p}...\ts{p}{p}T$ exists and at each $M$
    the limit $T_n(M)$ is a $T$-limit
  \item natural transformations $0:id \Rightarrow T$, $+:T\ts{p}{p}T \Rightarrow
    T$, $\ell:T \Rightarrow TT$ and $c: TT \Rightarrow TT$
  \end{itemize}
  such that:
  \begin{itemize}
  \item for every object $M$: $p_M$, $0_M$ and $+_M$ form an additive bundle
      B
  \item $(\ell, 0):(T(M), p) \rightarrow (T^2(M), T(p))$ is an additive bundle
    morphism
  \item $(c, id):(T^2(M), T(p)) \rightarrow (T^2(M), p)$ is an additive bundle
    morphism
  \item $c^2=id$, $c\ell = \ell$, $T(\ell)\ell = \ell\ell$, $cT(c)c = T(c)cT(c)$
    and $T(\ell)c = cT(c)\ell$
  \item the lift $\ell$ is universal: the diagram
    \[\begin{tikzcd}
        T_2(M) \rar{\ell\pi_0 +_{T(p)}0\pi_1} \dar{q\pi_0 = q\pi_1} &[3ex]
        T^2(M) \dar{T(p)}\\ M \rar{0} & T(M)
      \end{tikzcd}\] is a $T$-pullback.
  \end{itemize}
\end{definition}
\begin{definition}
  A \emph{tangent category with negatives} is a tangent category with a natural transformation $-: T \Rightarrow T$ which makes all of the commutative monoids ($p_M$, $0_M$, $+_M$)  abelian groups.
\end{definition}

Examples of tangent categories include the category of smooth manifolds and the infinitesimally linear objects in a model of synthetic differential geometry (see \cite{MR2244115}). For more examples of tangent categories,  see Example 2.2 in \cite{MR3792842}.

\subsection{Retractive display systems}
\label{sub:retractive-display}

In the category of smooth manifolds, the projection for every vector bundle is a submersion.
Submersions have useful $T$-stability properties - the $T$-pullback along any submersion exists and is itself a submersion, and they are stable under the tangent functor. Tangent display systems were introduced by Cockett and Cruttwell and axiomatise the class of submersions' $T$-stability properties in an arbitrary tangent category \cite{MR3792842}. %TODO Give the specific section in differential bundles
In this paper, we consider an extension to tangent display systems that are closed to retracts, which we call a retractive display system. 

\begin{definition}\label{def:display-system}
    A \textit{tangent display system} is a class of maps $\mathcal{D}$ that is:
    \begin{itemize}
        \item stable under $T$-pullbacks: the $T$-pullback along any map $d \in \mathcal{D}$ exists and is contained in $\mathcal{D}$,
        \item stable under the tangent functor.
    \end{itemize}
    We call any tangent display system that is closed to retracts in the arrow category a \textit{retractive} display system. 
    If for all $M$, $p_M \in \mathcal{D}$, we call $\mathcal{D}$ a proper (retractive) display system. 
\end{definition}
In this section, we shall show that the submersions in the category of smooth manifolds give a retractive display system, and give a general construction of retractive display systems from display systems. We recall the definition of a submersion:
\begin{definition}
   \label{def:submersion-sman}
   If $A$ and $B$ are smooth manifolds then a smooth function $f:A \rightarrow B$
   is a \emph{submersion} if and only if the derivative $Df|_a$ of $f$ at every point 
   $a\in A$ is a surjective linear map.
\end{definition}
In other words, $f$ is a submersion if and only if for all $a\in A$ and all $v\in T(B)$ such that $fa = pv$, there exists a $w\in T(A)$ such that $T(f)w = v$.
This is a \textit{weakly} universal cone over $A \xrightarrow{f} B \xleftarrow{p} TB$: there exists \textit{at least} one morphism into it for any other cone over the diagram. 
\begin{definition}\label{def:weak-pullback}
    We say that a commuting square is a \textit{weak pullback} if for any $x:X \to A$ and $y:X \to B$ so that $fx = gy$, there exists a map $X \to W$ making the following diagram commute:
    \[
        \begin{tikzcd}
        X \ar[bend left]{rrd}{x} \ar[bend right]{ddr}[swap]{y} \ar[dashed]{rd}{\exists} \\
            & W \rar{a} \dar[swap]{b} & A \dar{f} \\
            & B \rar[swap]{g} & C
        \end{tikzcd}
    \]    
\end{definition}
\begin{lemma}\label{lem:pb-retract}
    Should the pullback of $A \xrightarrow{f} C \xleftarrow{g} B$ exist, \cref{def:weak-pullback} is equivalent to asking the induced map $(a, b):W \to A \ts{f}{g} B$ be a split epimorphism.
\end{lemma}
We take the submersion property for a map $f$ using global elements (for all $a\in A$ and all $v\in T(B)$ such that $fa = pv$, there exists a $w\in T(A)$ such that $T(f)w = v$) and state it using generalized elements.
\begin{definition}\label{def:tangent-submersion}
    An arrow $f:A \rightarrow B$ in a tangent category is a \emph{tangent submersion} if and only if the naturality diagram
    \[
        \begin{tikzcd}
            TA \rar{Tf} \dar[swap]{p} & TB \dar{p} \\
            A \rar[swap]{f} & B
        \end{tikzcd}
    \] 
    is a weak $T$-pullback. 
\end{definition}{}
Following \cref{lem:pb-retract}, in the case the pullback exists this is equivalent to asking for a section $h: A \ts{f}{p} TB \to TA$ of the horizontal descent $(p,Tf): TA \to A \ts{f}{p} TB$ (this section is sometimes called a \textit{horizontal lift} in differential geometry literature \cite{MR980716}). 
In smooth manifolds, the $T$-pullback along the projection $p:T \Rightarrow id$ always exists, so to prove that every submersion is a tangent submersion it suffices to show the existence of a horizontal lift.
\begin{proposition}
    In the category of smooth manifolds, the tangent submersions are precisely the classical smooth submersions.
\end{proposition}
\begin{proof}
    There is an explicit construction of a horizontal lift for a classical smooth submersion in VII.1 of \cite{MR3242761}.
\end{proof}
It is possible to show that the $T$-stability properties for submersions in the category of smooth manifolds follow from the general theory of weak pullbacks.
We begin by showing that weak pullbacks satisfy a weakened version of the pullback lemma and then show that the retract of a weak pullback is a weak pullback (the first lemma may be found in \cite{MR980716}).
\begin{lemma}[Pullback lemma]\label{lem:weak-pullback}
    Consider the diagram:
    \[
        \begin{tikzcd}
            \bullet \rar \dar \ar[rd, phantom, "(A)"] & \bullet \rar{f}\dar{g} \ar[rd, phantom, "(B)"] & \bullet\dar \\
            \bullet \rar & \bullet \rar & \bullet 
        \end{tikzcd}
    \]
	If $f,g$ are jointly monic, then $(A)$ is a (weak) pullback if and only if the outer perimeter $(A) + (B)$ is a (weak) pullback. (Note that when $(B)$ is a pullback, $f,g$ are jointly monic.)
\end{lemma}    
\begin{proof}
    The proof for pullbacks holds for weak pullbacks.    
\end{proof}

\begin{lemma}\label{lem:wpb-retract}
    (Weak) pullbacks are closed to retracts.
\end{lemma}
\begin{proof}
    Suppose that $S'$ is a weak pullback, and $S$ is a retract of it in the category of commuting squares. Consider the following diagram (suppressing the subscripts for $s,r$):
    \[
        \begin{tikzcd}[column sep = small, row sep = small]
            Z \ar[bend right]{rrd} \ar[bend right]{ddr}  \ar[bend left, dashed]{rrr} & A \ar{rd}{x} \ar{dd}[swap, near end]{y} \ar{rr}{s} & & A'\ar{rd}{x'} \ar{dd}[swap, near end]{y'} \ar{rr}{r} & & A\ar{rd}{x} \ar{dd}[near end]{y} \\
            & & B \ar{dd}[near start]{w} \ar{rr}[near start]{s} 
            & & B' \ar{dd}[near start]{w'} \ar{rr}[near start]{r} & & B \ar{dd}[near start]{w} \\
            & C \ar{rd}[swap]{z} \ar{rr}[near start]{s} & & C' \ar{rd}[swap]{z'} \ar{rr}[near start]{r} & & C \ar{rd}[swap]{z}  \\
            & & D \ar{rr}{s} & & D' \ar{rr}{r}& & D
        \end{tikzcd}
    \]
    Given a cone for $S$, there is a corresponding cone for $S'$ which induces a map $Z \to A'$ and postcomposition with $r_A$ gives the desired map into $A$.
\end{proof}
Using these lemmas, it is straightforward to prove the following $T$-stability properties hold for tangent submersions.
\begin{lemma}\label{lem:submersion}
    In any tangent category $\mathbb{X}$:
    \begin{enumerate}[(a)]
        \item Tangent submersions are closed to composition.
        \item  Tangent submersions are closed to retracts. 
        \item Any $T$-pullback of a tangent submersion is a tangent submersion.
    \end{enumerate}
\end{lemma}
\begin{proof}
    (a) follows from \cref{lem:weak-pullback} while (b) follows from \cref{lem:wpb-retract}. It remains to prove (c):
    
    Suppose we have a $T$-pullback, where $u$ is a tangent submersion:
    \[
    \begin{tikzcd}
        A \rar{f} \dar{v} & M \dar{u} \\
        B \rar{g} & N
    \end{tikzcd}
    \]
    Then the following two diagrams are equal:
    \[
        \begin{tikzcd}
            TA \rar{Tf} \dar{Tv} & TM \rar{p} \dar{Tu} & M \dar{u} \\
            TB \rar{Tg} & TN \rar{p} & N 
        \end{tikzcd}
        =
        \begin{tikzcd}
            TA \rar{p} \dar{Tv} & A \rar{f} \dar{v} & M \dar{u} \\
            TB \rar{p} & B \rar{g} & N
        \end{tikzcd}
    \]
    the left diagram is a weak pullback by composition.  Therefore the outer perimeter of the right diagram is a weak pullback, and the right square is a pullback, so the left square is a weak pullback by the weak pullback lemma, as desired.
\end{proof}
We see that the class of tangent submersions is closed to retracts in the arrow category, and is conditionally closed under $T$-reindexing (if the $T$-pullback of a tangent submersion exists, it is a tangent submersion).
This leads to the following result:
\begin{proposition}\label{prop:display-submersions-are-r-display}
    Let $\mathbb{X}$ be a tangent category that allows for $T$-reindexing of the class of tangent submersions $\mathcal{R}$.
    Then the class of tangent submersions is a display system.
\end{proposition}
\begin{proof}
    Any class of maps that is closed to $T$-reindexing is a tangent display system, and the class of submersions is closed to retracts in the arrow category.
\end{proof}
In the category of smooth manifolds, where the class of smooth submersions is the canonical example of a proper tangent display system, this gives the following:
\begin{corollary}\label{cor:sman-r-display}
    The class of submersions in the category of smooth manifolds is a  proper retractive display system.
\end{corollary}
We can also specify this result to a tangent category with $T$-pullbacks.
\begin{corollary}\label{cor:submersions-are-descent}
    The split tangent submersions in a tangent category in which all pullbacks exist and are $T$-limits form a retractive display system. 
\end{corollary}{}

\begin{remark}\label{rem:enriched-wfs-sdg}
    In synthetic differential geometry, the weak pullback used to define tangent submersions is equivalent to a weak lifting property. Recall that Day \cite{MR0396717} characterised the orthogonal lifting property between two maps $f \bot g$ in $\mathbb{X}$ as a pullback in $\mathsf{Set}$:
    \[
        \forall u,v:
        \begin{tikzcd}
            A \rar{u} \dar[swap]{f} & X \dar{g} \\
            B \rar{v} \ar[dashed]{ru}{\exists!} & Y
        \end{tikzcd}
        \iff 
        \begin{tikzcd}
            \mathbb{X}(B,X) \rar{\mathbb{X}(id,g)} \dar[swap]{\mathbb{X}(f,id)} \ar[phantom]{rd}{pb}
             & \mathbb{X}(B,Y) \dar{\mathbb{X}(f,id)} \\
            \mathbb{X}(A,X) \rar[swap]{\mathbb{X}(id,g)} & \mathbb{X}(A,Y)
        \end{tikzcd}
    \]
    For a monoidal closed category, this can be strengthened to a \textit{self-enriched} orthogonal factorisation system by using the internal hom $[-,-]$. For a weak factorisation system, the weak lifting property (where the requirement that the lift is unique is dropped) is equivalently characterised by asking the commuting diagram on the right be a weak pullback.
    
    In synthetic differential geometry (or any \textit{representable tangent category}), the tangent functor is represented by pointed infinitesimal object $0: 1 \to D$, so that $p_M = [0,M]: [D,M] \to [1,M] \cong M$. Thus the condition that the naturality square for $p$ at $f$ be a weak pullback may be reinterpreted as a weak lifting property:
    \[
        \forall m,\gamma:
        \begin{tikzcd}
            1 \dar[swap]{0} \rar{m} & M \dar{q} \\
            D \rar{\gamma} \ar[dashed]{ru}{\exists}& N    
        \end{tikzcd}
        \iff
        \begin{tikzcd}
            {[D,M]} \rar{[id,f]} \dar[swap]{[0,id]} \ar[phantom]{rd}{wpb} & {[D,N]} \dar{[0,id]} \\
            {[1,M]} \rar[swap]{[id,f]} & {[1,N]}
        \end{tikzcd}
    \]
%     The category of microlinear spaces in a Grothendieck model of synthetic differential geometry is a locally presentable category, thus the tangent submersions are the right class of maps for a weak factorisation system by the enriched small-object argument. 
%     One can then use properties of weak factorisation systems and the well-adapted embedding of the category of smooth manifolds into the Dubuc topos (which, in particular, preserves pullbacks along submersions) to give a more sophisticated proof of \cref{cor:sman-r-display}.
\end{remark}

\section{Differential bundles}
\label{sec:differential_bundles}

In this section, we review the basic theory of differential bundles, and provide some new characterisations which we shall use throughout the rest of this paper. In the first section, we pull out the purely equational fragment of the definition of differential bundles, which we call pre-differential bundles, to simplify the definition of differential bundles. In the second section, we introduce differential bundles and show they are precisely pre-differential bundles satisfying a universal property.

\subsection{Pre-differential bundles}
\label{sub:pre-differential-bundles}

The original definition of a differential bundle consisted of an additive bundle and a lift satisfying various coherences and universal properties. 
In this section, we pull out the purely equational fragment of the definition of a differential bundle, which we call a pre-differential bundle. 
We begin with some basic theory regarding pre-differential bundles and how they relate to coalgebras of a weak comonad.

\begin{definition}
  \label{def:pre-differential-bundle}
  A \emph{pre-differential bundle} is a triple $(q:E \to M, \xi:M \to E, \lambda: E \to TE)$ so that $q\xi = id$,  $\ell\lambda = T(\lambda)\lambda$, $p\lambda = \xi q$, and $\lambda\xi = 0_E\xi$.
\end{definition}

From the definition of a tangent category, we have that $\ell\ell=T(\ell)\ell$, so that the pair $(T,\ell)$ is a \textit{weak} comonad. 
(For the definition of weak comonad see for instance 1.1 of \cite{MR3100133}.)
The condition that $T(\lambda)\lambda = \ell T(\lambda)$ is precisely the same as the requirement that $\lambda$  be a coalgebra of the weak comonad $(T,\ell)$. In fact if $\lambda$ is a coalgebra of $(T,\ell)$ then $p\lambda$ is an idempotent:
    \[ p \lambda p \lambda
        = p p T(\lambda) \lambda \\
        = p p \ell \lambda \\
        = p 0 p \lambda \\
        = p \lambda\]
and the condition that $\xi q = p\lambda$ states that $\xi,q$ is a splitting of the idempotent $p\lambda$.
This leads us to following proposition
\begin{proposition}
    A pre-differential bundle is precisely a coalgebra of $(T,\ell)$ equipped with a chosen idempotent splitting $\xi q=p\lambda$    
\end{proposition}
\begin{proof}
	We have checked that $p\lambda$ is always an idempotent, and $\xi,q$ splits it by definition. All that remains is to check the condition that $0\xi = \lambda\xi$.
	\[
        \lambda\xi 
        = \lambda\xi q \xi 
        = \lambda p \lambda \xi 
        = p T(\lambda) \lambda \xi 
        = p \ell \lambda \xi 
        = 0p \lambda \xi 
        = 0 \xi q \xi 
        = 0\xi
    \]
\end{proof}
There is a naturally defined category of pre-differential bundles, whose morphisms are pairs of maps $(f_1,f_0)$ which preserve the chosen idempotent splitting and preserve the coalgebra structure.
\begin{definition}
    The category of pre-differential bundles in a tangent category $\mathbb{X}$ has:
\begin{itemize}
    \item Objects: pre-differential bundles $(q:E \to M, \xi, \lambda)$
    \item Morphisms: A map $f: (q: E \to M,\xi,\lambda) \to (q': E' \to M', \xi', \lambda')$ is given by a pair of maps $f_1:E \to E', f_0:M \to M'$ so that the following diagrams commute. 
    \[
        \begin{tikzcd}
            E \rar{f_1} \dar[swap]{q}& E' \dar{q'} \\
            M \rar[swap]{f_0} & M'
        \end{tikzcd}
        \hspace{1cm}        
        \begin{tikzcd}
            TE \rar{Tf_1} & TE' \\
            E \rar[swap]{f_1} \uar{\lambda} & E' \uar[swap]{\lambda'}
        \end{tikzcd}
    \]
\end{itemize}
\end{definition}
We can see that the base maps are redundant data, so we can treat the category of differential bundles as a category of coalgebras over a weak comonad with extra data (the chosen splitting of $p\lambda$). 
\begin{proposition}
    The category of pre-differential bundles in $\mathbb{X}$ is isomorphic to the category of:
    \begin{itemize}
        \item Objects: pre-differential bundles $(q:E \to M, \xi, \lambda)$.
        \item Morphisms: a map $f: (q:E \to M, \xi, \lambda) \to (q':E' \to B', \xi', \lambda')$ is a coalgebra morphism $f: \lambda \to \lambda'$.
    \end{itemize}
\end{proposition}
\begin{proof}
    It suffices to show that given a morphism of differential bundles $(f_1, f_0)$, the morphism between base spaces is determined by $f_1$, which is immediate by the diagram:
    \[
        \begin{tikzcd}
            & E \rar{f_1} \dar{q} & E' \dar{q'} \\
            M \ar{ru}{\xi} \ar[equals]{r} & M \rar[swap]{f_0} & M'
        \end{tikzcd}
    \]
    thus $f_0 = q'f_1\xi$, and a linear bundle morphism is completely determined by a coalgebra morphism $f_1: \lambda \to \lambda'$.
\end{proof}

\subsection{Differential bundles}
\label{sub:alternative_differential_bundle}

In this section, we deconstruct the definition of a differential bundle to show that it is precisely a pre-differential bundle satisfying a universal property, expressed as a single $T$-pullback diagram.
We begin by considering Rosick\'{y}'s original universality diagram for the vertical lift on the tangent bundle \cite{MR800500}, and show that any pre-differential bundle satisfying this diagram has an induced addition map that satisfies the same coherences with the lift that $+$ and $\ell$ satisfy on the tangent bundle. 

\begin{lemma}\label{lem:prediff-addition}
Let $(q:E \to M,\xi, \lambda)$ be a pre-differential bundle in a tangent category such that:
\begin{itemize}
  \item $n$-fold pullback powers of $q$ exist and are $T$-limits,
  \item (Rosick\'{y}'s universality diagram) the commuting square
	  \begin{equation*}
	    \begin{tikzcd}
	      E \rar{\lambda} \dar{q} & TE \dar{(T(q), p)} \\
	      M \rar{(\xi, 0)} &[3ex] T(M) \times E
	    \end{tikzcd}
	  \end{equation*}
    is a $T$-pullback,
\end{itemize}
    then there is a uniquely determined addition $+_q: E_2 \to E$ making $(q,+_q,\xi)$ an additive bundle and $(\lambda,\xi): q \Rightarrow p$ and $(\lambda,0): q \Rightarrow T(q)$ additive bundle morphisms. 
    Furthermore, when the tangent category has negatives, the additive bundle will have negatives.
\end{lemma}
\begin{proof}
  First, note that $T^n(\lambda)$ is a monomorphism because it is the pullback of a monomorphism.
  The addition $+_q$ is defined using the following commutative diagram:
  \begin{equation}\begin{tikzcd}\label{eq:addition-via-universality}
      E_2 \arrow{ddd}[swap]{q\pi_0=q\pi_1} \drar[dashed]{+_q}
      \arrow{rrr}{\lambda\times \lambda} & &
      & T_2(E) \dlar[swap]{+_p} \arrow{ddd}{(T(q)\times T(q), p\pi_0)}\\
      & E \rar{\lambda} \dar[swap]{q} & T(E) \dar{(T(q), p)} & \\
      & M \rar[swap]{(0, \xi)} & T(M)\times E & \\
      M \urar{id} \arrow{rrr}[swap]{((0, 0), \xi)} & & & T_2(M)\times E \ular[swap]{+_p \times id}
    \end{tikzcd}\end{equation}
  and so in particular $+_q$ is the unique addition on $q$ such that $\lambda +_q = +_p(\lambda\times \lambda)$.
  The left-hand square gives the identity $q(a+_q b)=qa =qb$. Post-composition by $\lambda$ gives the associativity, commutativity and unit laws.

  Now that we have constructed an additive bundle structure, we must show that the bundle morphisms $(\lambda,\xi)$ and $(\lambda, 0)$ are additive. 
  We observe that $\lambda (a +_q b) = \lambda a +_p \lambda b$ by construction and $p\lambda = \xi q$ because the middle square commutes.
  Thus $(\lambda, \xi): (E,q) \to (TE,p)$ is a morphism of additive bundles.
  To show that $(\lambda, 0): (E,q) \to (TE, T(q))$ is additive first observe that $T(q)\lambda = 0q$ because the middle square commutes.
  To show $\lambda$ preserves addition, compute:
  \begin{align*}
    T(\lambda)(\lambda a +_{Tq} \lambda b)
    &= T(\lambda)\lambda a +_{Tp} T(\lambda)\lambda b \\
    &= l\lambda a +_{Tp} l\lambda b \\
    &= l(\lambda a +_p \lambda b) \\
    &= l\lambda (a +_q b)\\
    &= T(\lambda)\lambda (a +_q b)
  \end{align*}
  where $a,b\in E$ such that $qa=qb$.

  In case the tangent category has negatives, we may induce the map $-_q$ from $-_{p}$ via the diagram:
  \[
   \begin{tikzcd}
       E \ar{rrr}{\lambda} \ar{ddd}[swap]{q} \ar[dashed]{rd}{-_q}
     & & & TE \ar{ddd}{(Tq,p)} \ar{ld}[swap]{-_{PE}}\\
     & E \rar{\lambda} \dar[swap]{q} & TE \dar{(Tq,p)}\\
     & M \rar[swap]{(0,\xi)} & TM \times E \\
       M\ar{rrr}[swap]{(0,\xi)} \ar{ru}{-_{PE}} & & & TM \times E \ar{lu}[swap]{-_{PM},id}
   \end{tikzcd}
  \] 
  and postcomposition with $\lambda$ shows the necessary equations hold. 
\end{proof}

Thus, we can see that Rosick\'{y}'s universality diagram uniquely determines the additive bundle structure in a differential bundle.  
We now give a proof that a morphism of pre-differential bundles preserves addition.
\begin{proposition}\label{prop:linear-is-additive}
    Let $(q:E \to M, \xi, \lambda), (q',\xi',\lambda')$ be a pair of differential bundles satisfying Rosick\'{y}'s universality diagram.
    Then any coalgebra morphism $f: \lambda \to \lambda'$ gives rise to an additive bundle morphism $(f, q'f\xi) = (f,f')$.
\end{proposition}
\begin{proof}
    Note that $q',\lambda'$ are jointly monic, then check post-composition for $q'$
    \begin{gather*}
    	q'(f\pi_0 +_{q'} f\pi_1) = q'f\pi_0 = f'q\pi_0 = f'q(\pi_0 +_q \pi_1) = q'f(\pi_0 +_q \pi_1)
    	\end{gather*}
    Now post-composition by $\lambda'$
    \begin{gather*}
    	\lambda'(f \pi_0 +_{q'} f\pi_1) = \lambda'f\pi_0 +_{p} \lambda'f\pi_1 = T(f)\lambda\pi_0 +_p T(f)\lambda\pi_1 \\= T(f)(\lambda\pi_0 +_p \lambda\pi_1) = T(f)\lambda(\pi_0 +_p \pi_1) = \lambda'f(\pi_0 +_p \pi_1)
    \end{gather*}
    Therefore $(f,f')$ is an additive bundle morphism.
\end{proof}
%We have now seen that a pre-differential bundle will have a coherent additive bundle structure whenever it satisfies Rosick\'{y}'s universality diagram and $T$-pullbacks of the projection exist.
%We now move on to differential bundles. 
%At first glance, it would seem that differential bundles include an additive bundle structure that is missing from the definition of a pre-differential bundle. However, we shall see in \cref{sub:alternative_differential_bundle} that this is a property of a pre-differential bundle.
We now give the original definition of a differential bundle. 
A differential bundle is an additive bundle with a lift satisfying the same coherences with addition and universality conditions as the universal lift $\ell$.
Based on \cref{lem:prediff-addition}, we shall show that the universality conditions on $\ell$ induce an addition map. 
Thus, a differential bundle is a pre-diffential bundle satisfying some additional properties rather than having additional structure.
\begin{definition}
  A \emph{differential bundle} in a tangent category consists of arrows $q: E
  \rightarrow M$, $+_q:E_2 \rightarrow E$, $\xi:M \rightarrow E$ and $\lambda:E
  \rightarrow T(E)$ such that:
  \begin{itemize}
  \item $(q,\xi,\lambda)$ is a pre-differential bundle,
  \item $n$-fold pullbacks of $q$ exist and are $T$-limits,
  \item $(\lambda, 0):(E, q) \rightarrow(T(E), T(q))$ is an additive bundle
    morphism,
  \item $(\lambda, \xi):(E, q) \rightarrow (T(E), p)$ is an additive bundle
    morphism,
  \item (Cockett-Cruttwell universality) the square
    \begin{equation*}\begin{tikzcd}
        E_2 \rar{\lambda\pi_0 +_{T(q)}0\pi_1} \dar[swap]{q\pi_0=q\pi_1} &[7ex] T(E) \dar{T(q)}\\
        M \rar{0} & T(M)
      \end{tikzcd}\end{equation*}
  is a $T$-pullback.
  \end{itemize}
\end{definition}

The following proposition presents our new characterisations of differential bundles.
The first is essentially the same as the original definition; the only difference is that it uses Rosick\'{y}'s universality diagram to induce an additive bundle structure. 
The second uses a rather opaque pullback diagram, that is then related to Rosick\'{y}'s and the original Cockett-Cruttwell universality condition.
\begin{proposition}
 \label{prop:new-def-diff-bundle-arbitrary}
     The following are equivalent
     \begin{enumerate}
         \item $(q, \xi, \lambda)$ is a differential bundle.
         \item  $(q, \xi, \lambda)$ is a pre-differential bundle, all $T$-pullback powers of $q$ exist and the diagrams:
                     \begin{equation}
                             \label{eq:ros-universal}
                             \begin{tikzcd}
                              E \rar{\lambda} \dar{q} & TE \dar{(Tq,p)} \\
                              M \rar{(0,\xi)} & TM \times E
                             \end{tikzcd} 
                     \end{equation}
                     \begin{equation}
                         \label{eq:cockett-universal}
                         \begin{tikzcd}
                              E_2 \rar{\lambda \pi_0 +_{Tq} 0\pi_1} \dar[][description]{q\pi_0=q\pi_1} 
                              &[5em] TE \dar{Tq} \\
                              M \rar{0} & TM 
                         \end{tikzcd}
                     \end{equation}
                   are $T$-pullbacks (where the additive bundle structure on $T(q)$ is induced by \cref{lem:prediff-addition}).
         \item $(q, \xi, \lambda)$ is a pre-differential bundle, all $T$-pullback powers of $q$ exist and the diagram:
                   \begin{equation}\label{eq:new-universality}\begin{tikzcd}[column sep=13ex]
                     E_2 \rar{T(\lambda)\lambda\pi_0 +_{T(p)}T(\lambda)0\pi_1}
                     \dar[][description]{q\pi_0=q\pi_1} & T^2(E) \dar{(T^2(q), T(p))}\\
                     M \rar{(00, 0\xi)} & T^2(M)\times T(E)
                   \end{tikzcd}\end{equation}
                   is a $T$-pullback.
     \end{enumerate}
 \end{proposition}
 \begin{proof}
    We prove the chain of equivalences holds.      
    \begin{itemize}
        \item[$(1 \Rightarrow 2)$]
              We need only show that \cref{eq:ros-universal} is a $T$-pullback.
            If $a:A \rightarrow T^{n+1}(E)$ satisfies $T^{n+1}(q)a = T^n(0qp)a$
            and $T^n(p)a = T^n(\xi q p)a$ then by the universality of the lift there exists $(u, v):A \rightarrow
            T^n(E_2)$ such that
            \[T^n(\lambda)u +_{T^{n+1}(q)}T^n(0)v = a\] and $T^n(q)u = T^n(q)v =
            T^n(qp)a$. However by postcomposing the displayed equation with $T^n(p)$ we
            obtain that actually $v = T^n(p)a = T^n(\xi qp)a$ and so the displayed
            equation is equivalently $T^n(\lambda)u = a$. But this means that $u:A
            \rightarrow T^n(E)$ is the factorisation we require.
                Moreover $T^n(\lambda)$ is a monomorphism because $T^n(\mu)$ is a monomorphism and so the factorisation is unique. 
           \item[$(2 \Rightarrow 1)$] By \cref{lem:prediff-addition}, we have a uniquely determined additive bundle so that $(+_q,q,\xi)$, and $(\lambda,\xi): q \to p, (\lambda,0): q \to Tq$ are additive bundle morphisms. 
         The universality of the vertical lift holds for the induce addition by our assumption, thus $(\lambda, q, \xi)$ is a differential bundle.
        \item[$(3 \Rightarrow 2)$]
         We first exhibit Rosick\'{y}'s universality diagram \cref{eq:ros-universal} as a retract of the diagram \cref{eq:new-universality}:
       \begin{equation}\label{eq:old-lift-from-new}
         \begin{tikzcd}
           E_2 \ar{rrr}{\ell\lambda\pi_0+_{T(p)}0\lambda\pi_1}  \ar{ddd}[swap]{q\pi_0=q\pi_1} \drar{\pi_1}
           & & & T^2(E) \dlar[swap]{p} \arrow{ddd}{(T^2(q), T(p))}\\
           & E  \dar{q} \rar{\lambda} & T(E) \dar{(T(q), p)} & \\
           & M \rar{(0, \xi)} & T(M)\times E  & \\
           M \arrow{rrr}{(00, 0\xi)} \urar{id} & & & T^2(M)\times T(E)\ular[swap]{p}
         \end{tikzcd}
       \end{equation}
       Pullbacks are closed to retracts, thus Rosick\'{y}'s universality condition holds (\cref{eq:ros-universal} is a pullback)
       and we may induce an additive bundle structure as in \cref{lem:prediff-addition}.
       
       Now observe that in the diagram
       \begin{equation}\label{eq:pasting-for-universality}
       \begin{tikzcd}
         E_2 \rar{\lambda\pi_0 +_{Tq} 0\pi_1} \dar{q\pi_i}
         &[4ex] TE \rar{T\lambda} \dar{Tq} & T^2(E) \dar{(T^2q, Tp)}\\
           M \rar{0} & TM \rar{(T0,T\xi)} & T^2M \times TE
       \end{tikzcd}
       \end{equation}
       the outer perimeter is \cref{eq:new-universality} and the right square is a $T$-pullback,
       so by the pullback lemma the left square is a $T$-limit.
        
        \item[$(2 \Rightarrow 3)$] This also follows from the pullback lemma - as \cref{eq:cockett-universal} and \cref{eq:ros-universal}
       are $T$-pullbacks the composite \cref{eq:pasting-for-universality} is a $T$-pullback. 
    \end{itemize}
\end{proof}

When a tangent category has negatives, a pre-differential bundle that satisfies Rosick\'{y}'s universality diagram is a differential bundle.
\begin{corollary}\label{cor:new-def-diff-bundle-with-negatives}
  In a tangent category with negatives, a pre-differential bundle $(q: E \to M, \xi, \lambda)$ is a differential bundle if and only if
  $n$-fold $T$-pullback powers of $q$ exist and the diagram
          \begin{equation*}
            \begin{tikzcd}
              E \rar{\lambda} \dar{q} & T(E) \dar{(T(q), p)} \\
              M \rar{(\xi, 0)} &[3ex] T(M) \times E
            \end{tikzcd}
          \end{equation*}
   is a $T$-pullback.
\end{corollary}
\begin{proof}
    The reverse implication follows by \cref{prop:new-def-diff-bundle-arbitrary}, to prove the forwards implication
    it suffices to prove
    \[
    \begin{tikzcd}
         E_2 \rar{\lambda \pi_0 +_{Tq} 0\pi_1} \dar[][description]{q\pi_0=q\pi_1} 
         &[5em] TE \dar{Tq} \\
         M \rar{0} & TM 
    \end{tikzcd}
    \]
    is universal. Suppose that $a:A \rightarrow T^{n+1}(E)$ satisfies $T^{n+1}(q)a = T^n(0qp)a$.
  We need to show that there exists a unique factorisation of $a$ through $T^n(\mu):T^n(E_2) \rightarrow T^{n+1}(E)$.
  The difference $a-_{T^{n+1}(q)}T^n(0p)a:A \rightarrow T^{n+1}(E)$ satisfies:
  \begin{itemize}
  \item $T^{n+1}(q)(a-_{T^{n+1}(q)}T^n(0p)a)=T^{n+1}(q)a = T^n(0qp)a$ and
  \item $T^n(p)(a-_{T^{n+1}(q)}T^n(0p)a)=T^n(p)a-_{T^n(q)}T^n(p)a=T^n(\xi qpa)a$
  \end{itemize}
  so there exists an $e:A \rightarrow T^n(E)$ such that $T^n(\lambda)e =
  a-_{T(q)}T^n(0p)a$ and $T^n(q)e=T^n(qp)a$. Now $(e, T^n(p)a):A \rightarrow
  T^n(E_2)$ is the factorisation we require:
  \begin{align*}
  T^n(\mu)(e, T^n(p)a) &=T^n(\lambda\pi_0+_{T(q)}0\pi_1)(e, T^n(p)a)\\
  &=T^n(\lambda)e+_{T^{n+1}(q)}T^n(0p)a\\
                                                &=(a-_{T^{n+1}(q)}T^n(0p)a)+_{T^{n+1}(q)}T^n(0p)a = a
  \end{align*}
  and this factorisation is unique because $T^n(\lambda)$ and $T^n(0)$ are monomorphisms.
\end{proof}
Thus, differential bundles are pre-differential bundles satisfying a universality condition, so they are a full-subcategory of pre-differential bundles. 
\begin{definition}\label{def:cat-of-differential-bundles}
	The category of differential bundles is the full subcategory of the category of pre-differential bundles satisfying the Rosick\'{y} and Cockett-Cruttwell universality conditions.
\end{definition}
\begin{lemma}
	Every morphism of differential bundles preserves addition.
\end{lemma}
\begin{proof}
	Every differential bundle is a pre-differential bundle that satisfies Rosick\'{y}'s universality diagram, so this follows immediately from \cref{prop:linear-is-additive}.
\end{proof}

\section{Relating differential bundles and tangent projections}\label{sec:relating-bundles-and-tangent-projections}

In this section, we prove the general results from which we deduce that every differential bundle in the category of smooth manifolds is a vector bundle.
Our general strategy is to demonstrate that a differential bundle $(q: E \to M, \xi, \lambda)$ is a linear retract of a pullback of the tangent bundle on $E$.
In fact we prove this twice using two different sets of assumptions.
On the one hand in \cref{sub:universal_retraction} we work in a tangent category with negatives in which the pullback of the differential bundle $p:T(E) \rightarrow E$ along $\xi$ exists. 
On the other hand in \cref{sub:differential-bundles-as-retracts} we work in a general tangent category and assume that $(q, \xi, \lambda)$ is a \emph{strong differential bundle} (defined in \Cref{def:strong-differential-bundle}) of which all differential bundles in a tangent category with negatives are examples.
Then in \Cref{sub:differential-bundles-retractive-display} we use retractive display systems (see \cref{sub:retractive-display}) to characterise differential bundles as pre-differential bundles satisfying a single pullback diagram.

\subsection{Differential bundles as retracts}
\label{sub:universal_retraction}

In this section, we consider a fixed differential bundle $(q: E \rightarrow M, \xi, \lambda)$ in a tangent category $\mathbb{X}$ with negatives.
Furthermore we assume that the following pullback
\[\begin{tikzcd}
  T_M(E) \rar{\iota_M} \dar[swap]{\pi_M} & T(E) \dar{p}\\
  M \rar[swap]{\xi} & E
\end{tikzcd}\]
exists in $\mathbb{X}$, $\pi_M$ is a differential bundle and $(\iota_M, \xi)$ is a linear morphism of differential bundles.
Under the above assumptions we prove that $q$ is a retract of the pullback $\pi_M$:
\[\begin{tikzcd}
T_M(E) \rar[twoheadrightarrow]{K} \dar{\pi_M} & E \rar[rightarrowtail]{(\lambda, p)} \dar{q} & T_M(E) \dar{\pi_M}\\
M \ar[equals]{r} & M \ar[equals]{r} & M
\end{tikzcd}\]
where the map between the base spaces is the identity.
First, we describe the idempotent whose splitting defines this retract.

\begin{lemma}
  The arrow $\chi:T_M(E) \rightarrow T_M(E)$ uniquely determined by the equation
  \[\iota_M\chi = \iota_M-_p T(\xi q)\iota_M\]
  is idempotent and linear.
\end{lemma}
\begin{proof}
  The sum on the right hand side is well-typed because
  \[
    pT(\xi q)\iota_M = \xi q p\iota_M = \xi q \xi \pi_M = \xi\pi_M = p\iota_M
  \]
  and the arrow $\chi$ is factors through $T_M(E)$ because $p\iota_M \chi = p\iota_M = \xi
  \pi_M$. To see that $\chi$ is idempotent:
  \begin{align*}
    \iota_M\chi\chi &= (\iota_M-_p T(\xi q)\iota_M)\chi\\
    &= \iota_M\chi-_p T(\xi q)\iota_M\chi\\
    &= (\iota_M-_p T(\xi q)\iota_M)-_p T(\xi q)(\iota_M -_p T(\xi q)\iota_M)\\
    &=(\iota_M-_p T(\xi q)\iota_M) = \iota_M\chi
  \end{align*}
  and so $\chi\chi = \chi$ because $\iota_M$ is a monomorphism. To see that
  $\chi$ is linear:
  \begin{align*}
    T(\iota_M)\lambda_M\chi &= \ell \iota_M\chi\\
        &= \ell(\iota_M-_p T(\xi q)\iota_M)\\
        &=\ell\iota_M -_{T(p)} T^2(\xi q)\ell\iota_M\\
        &= T(\iota_M)\lambda_M -_{T(p)} T^2(\xi q)T(\iota_M)\lambda_M\\
        &= (T(\iota_M)-_{T(p)} T^2(\xi q)T(\iota_M))\lambda_M = T(\iota_M)T(\chi) \lambda_M
  \end{align*}
  and so $\lambda_M \chi = T(\chi) \lambda_M$ because $T(\iota_M)$ is a monomorphism.
  Note that $T(\iota_M)\lambda_M = \ell\iota_M$ because $(\iota_M, \xi)$ is a linear morphism of differential bundles.
\end{proof}

Next, we show that the idempotent $\chi$ splits.

\begin{proposition}\label{prop:bundle-as-splitting}
  If $q:E \rightarrow M$ is a differential bundle in a tangent category with negatives and the pullback $T_M(E)$ exists then 
  \[\begin{tikzcd}
      E \rar[rightarrowtail]{(\lambda, q)} & T_M(E) \rar[yshift=1ex]{\chi}
      \rar[yshift=-1ex][swap]{id} & T_M(E)
    \end{tikzcd}\] is an equaliser that is preserved by any functor.
\end{proposition}
\begin{proof}
  Let $a:A \rightarrow T_M(E)$ such that $\chi a = a$. First we show that there exists $e:A \rightarrow E$ such that $(\lambda, q)e = a$. To this end observe
  that
  \[T(q)\iota_M a = T(q)\iota_M \chi a = T(q)(\iota_M-_p T(\xi q)\iota_M) a =
    0 q p\iota_M a\]
  and also $p\iota_M a = \xi q p \iota_M a$ by the definition of $T_M(E)$.
  Therefore by the universality of the vertical lift as described in \Cref{cor:new-def-diff-bundle-with-negatives} there exists a unique $e:A
  \rightarrow E$ such that $\iota_M a = \lambda e$. This is the arrow $e$ that
  we require:
  \[
    \iota_M(\lambda, q)e = \lambda e = \iota_M a
  \]
  and so $(\lambda, q)e = a$ because $\iota_M$ is a monomorphism. To check this
  is unique let $f:A \rightarrow E$ be an arrow such that $(\lambda, q)f = a$. Then
  \[
    \iota_M(\lambda, q)f = \iota_M a\implies \lambda f= \lambda e \implies f=e
  \]
  and so $(\lambda, q)$ is the equaliser of $\chi$ and $id$.
  Since $\chi$ is idempotent the equaliser $(\lambda, q)$ is preserved by any functor.
\end{proof}

\begin{corollary}\label{cor:splitting-of-differential-bundles}
  If $K$ is the factorisation of $\chi$ through $(\lambda, p)$ then
  \[\begin{tikzcd}
      T_M(E) \rar[twoheadrightarrow]{K} \dar{\pi_M} & E \dar{q} \rar[rightarrowtail]{(\lambda, q)} & T_M(E) \dar{\pi_M}\\
      M \ar[equals]{r} & M \ar[equals]{r} & M
    \end{tikzcd}\] is a retract over the fixed base space $M$.
\end{corollary}

\begin{corollary}\label{cor:negatives-pullback-retract}
	If $(q:E \to M, \xi, \lambda)$ is a differential bundle and the pullback differential bundle $\pi_M$ exists then $(q,\xi,\lambda)$ is a retract of a pullback of a tangent bundle.
\end{corollary}

\subsection{Strong differential bundles as retracts}
\label{sub:differential-bundles-as-retracts}

In this section, we re-interpret the results of \cref{sub:universal_retraction} in an arbitrary tangent category.
This process reveals a third universality condition satisfied by differential bundles in tangent categories with negatives, which states that the kernel of $p:TE \to E$ splits as a fibred biproduct $E \ts{q}{p} TM$. %TODO add reference
We call differential bundles satisfying this third universality condition strong, and show that in a tangent category with negatives any differential bundle $(q:E \to M, \xi, \lambda)$ is strong if the $T$-pullback $E \ts{q}{p} TM$ exists, from which \cref{cor:splitting-of-differential-bundles} follows.

\begin{definition}\label{def:strong-differential-bundle}
  A differential bundle $(q:E \rightarrow M, \xi, \lambda)$ is \emph{strong} if
\begin{equation*}\begin{tikzcd}
    E \ts{q}{p}TM \rar{\lambda\pi_0 +_{p}T(\xi)\pi_1} \dar[swap]{q\pi_0=p\pi_1} &[5ex] T(E) \dar{p}\\
    M \rar{\xi} & E
\end{tikzcd}\end{equation*}
is a $T$-pullback.
%Equivalently using the notation of \Cref{sub:universal_retraction}: there is an isomorphism $T_M(E) \cong E\ts{q}{p}T(M)$.
\end{definition}
We now make it rigorous that tangent vectors on $T$ splits into a biproduct. 
Whenever pullback powers of $E \ts{q}{p} TM \to M$ exist, then this bundle is a biproduct in the category of differential bundles above $M$.
\begin{lemma}\label{lem:strong-differential-bundle-biproduct}
  Let $\mathbb{X}$ be a tangent category and $E$ a strong differential bundle
  % added the condition here
  so that $T$-pullback powers of $q\pi_0:E \ts{q}{p} TM \to M$ exist.
  Then:
  \begin{enumerate}[(a)]
    \item $(q\pi_0:E\ts{q}{p} TM \rightarrow M, (\xi,0_M), \lambda\times\ell)$, is a differential bundle,
    \item $(q\pi_0 (\xi,0_M), \lambda\times\ell)$ is the biproduct of $(q,\xi,\lambda)$ and $(p_M,0_M,\ell_M)$ in the category of differential
    bundles and linear morphisms above $M$. 
  \end{enumerate}
\end{lemma}
\begin{proof}
~\begin{enumerate}[(a)]
	\item To check that $q\pi_0: E\ts{q}{p}TM \to M$ is a differential bundle use the same construction found in corollary 5.9 of \cite{MR3792842}.
	  In particular by \cref{cor:new-def-diff-bundle-with-negatives} it suffices to prove that the following diagram is
	  a $T$-limit 
	  \[
	    \begin{tikzcd}
	      (E \ts{q}{p} TM)_2 \rar{(\lambda \times l)\pi_0 +_{T(q\pi_0)} 0\pi_1} \dar{q\pi_0}
	      &[10ex] T(E \ts{q}{p} TM) \dar{T(q\pi_0)} \\
	      M \rar{(\xi,0)} & TM
	    \end{tikzcd}
	  \]
	  but this follows immediately by the commutation of limits.
	\item Now we show that $E \ts{q}{p} TM$ is a coproduct.
	  First note that $id_M: M \to M$ is the zero object in the category of differential bundles over $M$ and that each differential bundle $q: E \rightarrow M$ has a zero morphism $(\xi, id_M): 1_M \Rightarrow q$ which is preserved by any linear morphism.
	  The coproduct diagram is:
	  \[
	    E \xrightarrow{\iota_0 = (id, 0 q)} E \ts{q}{p} TM \xleftarrow{\iota_1 = (\xi p, id)}TM
	  \]
	  because for any pair of linear bundle morphisms $f: E \to Z, g: TM \to Z $ the map
	  $(f\pi_0 +_{Z} g\pi_1)$ satisfies:
	  \[
	    (f\pi_0 +_Z g\pi_1)(\xi p, id) = (f \xi p +_Z g id) 
	    = (\xi_Z q +_Z g) = g 
	  \]
	  and
	  \[
	    (f\pi_0 +_Z g \pi_1)(id, 0q) = (f\pi_1 +_Z g0q) 
	      = (f \pi_1 +_Z \xi_Z q) = f
	  \]
	  which are the equations expressing the universal property of a coproduct.
	  Next, we check the biproduct identities:
	  \begin{align*}
	    \pi_0\iota_0 &= \pi_0 (id, 0q) = id \\
	    \pi_1\iota_0 &= \pi_1 (id, 0q) = 0q \\
	    \pi_0\iota_1 &= \pi_0 (\xi p, id) = \xi p \\
	    \pi_1 \iota_1 &= \pi_1 (\xi p, id) = id
	  \end{align*}
	  and so $E\ts{q}{p}T(M)$ is the biproduct of $q:E \rightarrow M$ and $p:T(M) \rightarrow M$ in the category of differential bundles over $M$.
	\end{enumerate}
\end{proof}
In the previous section, we proved that in a tangent category with negatives in which the pullback along $p$ always exists, we could characterise the bundle $T_M(E)$ as an idempotent splitting.
The bundle $T_M(E)$ is the pullback of $p_E$ along $\xi:M \to E$, and a strong differential bundle characterises this pullback as a biproduct $TM \ts{p}{q}E$, so there is a canonical linear idempotent that splits as $(q:E \to M, \xi, \lambda)$. 
\begin{corollary}
    Let $(q:E \to M, \xi,\lambda)$ be a strong differential bundle in a tangent category, and assume $T$-pullback powers of $E \ts{q}{p} TM$ exist. 
    Then there is an idempotent splitting in the category of differential bundles:
    \[
        \begin{tikzcd}
            E \ts{q}{p} TM \rar{\pi_0} \dar[swap]{q\pi_0 = p\pi_1}& E \rar{\iota_0} \dar{q}& E \ts{q}{p} TM \dar{q\pi_0 = p\pi_1}\\
            M \ar[equals]{r} & M \ar[equals]{r} & M
        \end{tikzcd}
    \]
    Furthermore, if $T$-pullback powers of $E \ts{q}{p} TM \to M$ exist, this is a linear splitting in the category of differential bundles above $M$.
\end{corollary}

In a tangent category with negatives, any differential bundle will be strong, provided the $T$-pullback of the projection along the tangent projection exists.
This means the strong universality condition may be seen as ``the other side'' of the Cockett-Cruttwell universality condition, and both follow from Rosick\'{y}'s universality condition.
\begin{lemma}
  \label{lem:strong-differential-bundle-equivalent-pullbacks}
  If $(q:E \to M, \xi, \lambda)$ is a pre-differential bundle in a tangent category with negatives and the $T$-pullback $E\ts{q}{p}T(M)$ exists then
  \[\begin{tikzcd}
    E \rar{\lambda} \dar[swap]{q} & T(E) \dar{(p,T(q))} \\
    M \rar{(\xi, 0)} & E \times TM
    \end{tikzcd}\]
    is a $T$-pullback if and only if
    \[\begin{tikzcd}
      E \ts{q}{p} T(M) \rar{\lambda\pi_0 +_p T(\xi)\pi_1}
      \dar[swap]{q\pi_0} &[5ex] T(E) \dar{p} \\
      M \rar{\xi} & E
    \end{tikzcd}\]
  is a $T$-pullback.
\end{lemma}
\begin{proof}
  First, we prove the forward implication.
  So let $a: A
  \to T^{n+1}(E)$ so that $T^n(p)a = T^{n}(\xi qp)a$.
  Then $a -_{T^{n}p} T^{n+1}(\xi q)a$ satisfies:
  \begin{itemize}
  \item $T^{n+1}(q)(a -_{T^{n}p} T^{n+1}(\xi q)a) = T^{n+1}(q)a -_{T^{n}p}
    T^{n+1}(q\xi q) a = T^{n}(0qp)a$ and
  \item $T^{n}(p)(a -_{T^{n}p }T^{n+1}(\xi q)a) = T^{n}(p)a = T^{n}(\xi qp)a$
  \end{itemize}
  and so using the first bullet point in the statement of this lemma we induce an $a': A \to T^{n}E$ such that $T^n(\lambda)a' = a -_{T^{n}(p)} T^{n+1}(\xi q)a$. 
  Now $(a',
  T^n(q)a): A \to E \ts{q}{p} TM$ is the factorisation we need because:
  \begin{align*}
      (T^n(\lambda) \pi_0 +_{T^n(p)} T^{n+1}(\xi)\pi_1) (a', T^{n+1}(q)a) &= T^n(\lambda) a' +_{T^n(p)} T^{n+1}(\xi q)a \\
      &= (a -_{T^n(p)} T^{n+1}(\xi q) a) +_{T^n(p)} T^{n+1} (\xi q) a \\
      &= a
  \end{align*}
  which is unique because $T^n(\lambda)$ and $T^{n+1}(\xi)$ are monomorphisms.
  Next, we prove the converse implication.
  So let $a: A \to T^{n+1}E$ be such that $T^{n+1}(q)a =
  T^{n}(0 q p)a$ and $T^n(p)a = T^n(\xi qp)a$.
  By combining the latter equation with the assumption contained second bullet point of the statement of this lemma we induce a pair of maps
  $(\hat{a}_0,\hat{a}_1): A \to T^n(E \ts{q}{p}TM)$ so that
  \[
  T^n(\lambda)\hat{a}_0 +_{T^n(p)} T^{n+1}(\xi)\hat{a}_1 = a\]
  holds.
  Post-composing both sides of this equation with $T^{n+1}(q)$ gives $\hat{a_1} = T^{n+1}(q) a = T^n(0qp) a$ and so the displayed equation is in fact $T^n(\lambda)\hat{a_0} = a$ which shows that $\hat{a_0}:A \rightarrow T^n(E)$ is the factorisation we require.
  This factorisation is unique because $T^n(\lambda)\pi_0 +_{T^n(p)} T^{n+1}(\xi)\pi_1$ and $T^{n+1}(\xi)$ are monomorphisms.
%   \begin{align*}
%     T^n(\xi qp) a 
%     &= T^{n+1}(q) a \\
%     &= T^{n+1}(q)(T^n(\lambda)\hat{a}_0 +_q T^{n+1}(\xi)\hat{a}_1) \\
%     &= \hat{a}_1
%   \end{align*}
\end{proof}
\begin{corollary}
  \label{cor:strong-differential-bundle-negatives}
  Let $(q:E \to M,\xi,\lambda)$ be a pre-differential bundle in a tangent category with
  negatives such that:
  \begin{itemize}
  \item $n$-fold pullback powers of $q$ exist and are $T$-limits and
  \item the pullback $E \ts{q}{p}TM$ exists and is a $T$-pullback.
  \end{itemize}
  Then $(\lambda,q,\xi)$ is a strong differential bundle if and only if
  \item \[\begin{tikzcd}
      E \rar{\lambda} \dar[swap]{q} & T(E) \dar{(p,T(q))} \\
      M \rar{(\xi, 0)} & E \times TM
    \end{tikzcd}\]
    is a $T$-pullback.
\end{corollary}
Recall that in the category of smooth manifolds, $p_M$ is a submersion so $T$-pullbacks along $p_M$ exist for all $M$, thus we have the following corollary.
\begin{corollary}\label{cor:differential-bundles-in-smooth-manifolds-are-strong}
    In the category of smooth manifolds, every differential bundle is strong.
\end{corollary}

\subsection{Differential bundles and retractive display systems}\label{sub:differential-bundles-retractive-display}

The condition that a differential bundle $(q: E \to M, \xi, \lambda)$ has $T$-pullback powers of its projection $q$ is necessary to ensure it has a coherent additive bundle structure.
In \cite{MR3792842}, \textit{display} differential bundles were considered - these are differential bundles $(q:E \to M, \xi, \lambda)$ in a display tangent category $(\mathbb{X}, \mathcal{D})$ satisfying $q \in \mathcal{D}$, which guarantees the existence of $T$-pullback powers of $q$ (along with some convenient re-indexing properties).
In this section we explore how the strong universality condition of \cref{sub:differential-bundles-as-retracts} interacts with the retractive display systems of \cref{sub:retractive-display}, and find that every strong differential bundle is displayed.

\begin{definition}\label{def:display-diff-bun}
    Let $\mathbb{X}$ be a tangent category with a display system $\mathcal{D}$. 
    We say that $(q,\xi, \lambda)$ is a $\mathcal{D}$-displayed differential bundle if $q \in \mathcal{D}$. 
\end{definition}

Now suppose our tangent category has a proper retractive display system $\mathcal{R}$.
Because a strong differential bundle naturally splits a linear idempotent of a pullback of a tangent projection, every strong differential bundle is $\mathcal{R}$-display.
\begin{proposition}
\label{prop:strong-differential-bundles-closed-under-base-change-and-right-factor}
    If $\mathcal{R}$ is a proper retractive display system on a tangent category $\mathbb{X}$ 
    and $(q: E\to M, \xi, \lambda)$ is a \emph{strong} differential bundle in $\mathbb{X}$, then $q$ is in $\mathcal{R}$.
\end{proposition}
\begin{proof}
    This follows immediately from \cref{lem:strong-differential-bundle-biproduct}.
    Let $(q:E \to M, \xi, \lambda)$ be a strong differential bundle and consider the following diagrams:
    \[
        \begin{tikzcd}
            E \ts{q}{p} TM \rar{\nu} \dar[swap]{q\pi_0 = p\pi_1} & TE \dar{p} \\
            M \rar{\xi} & E
        \end{tikzcd}
        \hspace{0.5cm}
            \begin{tikzcd}
            E\ts{q}{p}T(M) \dar{q\pi_0} \rar{\pi_0} & E \dar{q} \rar{(id, 0q)} & E\ts{q}{p}T(M) \dar{q\pi_0} \\
             M \ar[equals]{r} & M \ar[equals]{r} & M
        \end{tikzcd}
    \]
    the left diagram exhibits $(q\pi_0, (\xi,0), (\lambda \times \ell))$ as the pullback of the tangent projection on the total space $p_E$, and the right diagram exhibits $(q,\lambda,\xi)$ is the splitting of the linear idempotent $\iota_0\pi_0$ on  $(q\pi_0, (\xi,0), (\lambda \times \ell))$, so we have $q$ is $\mathcal{R}$-displayed.
\end{proof}{}
We can use \cref{prop:strong-differential-bundles-closed-under-base-change-and-right-factor} to show that every differential bundles in a tangent category with negatives and a retractive display system is $\mathcal{R}$-displayed.
We first prove a lemma that holds in a tangent category with negatives and display system.
\begin{lemma}\label{lem:negatives-R-every-db-strong}
    In a tangent category with negatives and a proper display system $\mathcal{D}$, every differential bundle is strong.
\end{lemma}
\begin{proof}
    Consider a differential bundle $(q:E \to M, \xi, \lambda)$, by \cref{cor:strong-differential-bundle-negatives} this is a strong differential bundle if and only if the $T$-pullback $E \ts{q}{p} TM$ exists, but this holds as $p \in \mathcal{D}$.
\end{proof}
The following corollary is a straightforward application of the two previous results  (\cref{prop:strong-differential-bundles-closed-under-base-change-and-right-factor} and \cref{lem:negatives-R-every-db-strong}).
\begin{corollary}\label{cor:negatives-retractive-display-all-differential}
    In a tangent category with negatives and a proper retractive display system $\mathcal{R}$, every differential bundle $(q:E \to M, \xi,\lambda)$ has its projection $q \in \mathcal{R}$.
\end{corollary}
When we apply this corollary to the category of smooth manifolds, where $\mathcal{R}$ is the class of submersions, we have the following:
\begin{corollary}
    In the category of smooth manifolds where the class of submersions is $\mathcal{R}$, every differential bundle is $\mathcal{R}$-display.
\end{corollary}
Now that we have shown that the projection of a differential bundle in the category of differential bundles is a submersion, we may rewrite the lift in local coordinates.
\begin{proposition}\label{lem:lift-in-local-coords}
    Let $(q: E \to M, \xi, \lambda)$ be a differential bundle in the category of smooth manifolds. Then the lift $\lambda$ may be rewritten in local coordinates as $\lambda(u,a) = (u,0,0,a)$.
\end{proposition}
\begin{proof}
    We can use the implicit function theorem to write $\lambda$ in terms of local co-ordinates (see \cite{MR2680546} for details).
    The equations $T(q)\lambda = 0q$ and $p \lambda = \xi q$ imply that $\lambda(u, a) = (u, 0, 0, \Lambda(u, a))$ for some $\Lambda(u, a) \in \mathbb R^n$.
    Next we write the universality of the lift in terms of local coordinates:
    \[\begin{tikzcd}
    B^k\times \mathbb R^n \drar{\iota} \arrow{rrr}{(\pi_0, 0, 0, \pi_1)} \arrow{ddd}{\pi_0} & & & B^k\times \mathbb R^n\times \mathbb{R}^k\times \mathbb R^n \arrow{ddd}{(\pi_0, \pi_2, \pi_0, \pi_1)} \dlar{\iota}\\
    & E \dar{q} \rar{\lambda} & T(E) \dar{(T(q), p)} & \\
    & M \rar{(0, \xi)} & T(M)\times E & \\
    B^k \arrow{rrr}{(id, 0, id, 0)} \urar{\iota} & & & B^k\times \mathbb R^k \times B^k \times \mathbb R^n \ular{\iota}
    \end{tikzcd}\]
    where $B^k$ is the unit ball in $\mathbb R^k$ and so $\lambda(u, a) = (u, 0, 0, a)$.
\end{proof}
   
In case a tangent category with negatives has a proper retractive display system $\mathcal{R}$, every differential bundle is $\mathcal{R}$-display.
Because we have shown the equivalence of various universality conditions, we can use the retract-closed property to force various pullbacks to be $T$-pullbacks.
We use this to give a simplified definition of a differential bundle that is simpler in practice to verify.
\begin{lemma}
  \label{cor:descent-limit-iff}
  Let $(q:E \to M, \xi, \lambda)$ be a pre-differential bundle in a tangent category with negatives and a proper display system $\mathcal{D}$ (as defined in \cref{def:display-system}). If either of the diagrams in the statement of \cref{lem:strong-differential-bundle-equivalent-pullbacks} 
    \[\begin{tikzcd}
    E \rar{\lambda} \dar[swap]{q} & T(E) \dar{(p,T(q))} \\
    M \rar{(\xi, 0)} & E \times TM
    \end{tikzcd}
    \hspace{0.1cm}
    \begin{tikzcd}
      E \ts{q}{p} T(M) \rar{\lambda\pi_0 +_p T(\xi)\pi_1}
      \dar[swap]{q\pi_0} &[5ex] T(E) \dar{p} \\
      M \rar{\xi} & E
    \end{tikzcd}\]
  is a pullback then both diagrams are $T$-pullbacks.
\end{lemma}
\begin{proof}
    By \cref{lem:strong-differential-bundle-equivalent-pullbacks}, if one is a pullback the other is. If the strong universality diagram is a pullback it is a $T$-pullback because $p\in \mathcal{D}$. Then by \cref{cor:descent-limit-iff} both diagrams are $T$-pullbacks.
\end{proof}

The category of smooth manifolds is a tangent category with negatives.
Furthermore, the surjective submersions form a proper retractive display system $\mathcal{R}$ as defined in \cref{def:display-system}.
The following result describes a situation where the characterisation of differential bundles further simplifies. 
First, we remove the requirement that pullback powers of $q$ exist.
Second, we remove the requirement that the diagram expressing the universality of the lift is a $T$-limit (although we still require it be a limit).

\begin{proposition}
  \label{prop:diffential-bundle-as-pullback}
  Let $\mathbb{X}$ be a tangent category with negatives equipped with a proper retractive system $\mathcal{R}$. 
  A pre-differential bundle $(q:E \to M, \xi, \lambda)$ is a differential bundle if and only if 
  \begin{equation}
  \label{eq:negatives-descent-universality}
    \begin{tikzcd}
      E \rar{\lambda} \dar{q} & TE \dar{(p,Tq)} \\
      M \rar{(\xi,0)} & E \times TM
    \end{tikzcd} 
  \end{equation}
  is a pullback.
\end{proposition}
\begin{proof}
    The forward implication holds by \cref{cor:new-def-diff-bundle-with-negatives} so it remains to prove the reverse implication.
    First \cref{cor:descent-limit-iff} implies that both
    \[
        \begin{tikzcd}
            E \ts{q}{p} TM \rar{\nu} \dar[swap]{q\pi_0 = p\pi_1} & TE \dar{p} \\
            M \rar{\xi} & E
        \end{tikzcd}
        \hspace{0.5cm}
            \begin{tikzcd}
                 E \rar{\lambda} \dar[swap]{q} & TE \dar{p} \\
                 M \rar{(\xi,0)}& E \times TM
        \end{tikzcd}
    \]
    are $T$-limits. Then $q \in \mathcal{R}$ because $\mathcal{R}$ is closed to retracts, so $T$-pullback powers of $q$. By \cref{cor:new-def-diff-bundle-with-negatives}, $(q,\xi,\lambda)$ is a differential bundle.
\end{proof}

The category of smooth manifolds has a retractive display system - the class of smooth submersions - so, in particular, we have the following corollary.
\essentialAlgebraicDescent

%The goal of this paper is to show that every differential bundle in the category of smooth manifolds is a vector bundle.
%Using \cref{cor:descent-limit-iff}, that a pre-differential bundle satisfying with an additive bundles structure is insufficient, thus demonstrating the necessity of the universality conditions on differential bundles.
%TODO rework thisMR3821920
\section{Differential bundles in smooth manifolds}
\label{sec:differential_bundles_are_vector_bundles}

In this section we prove that the category $VBun$ of (smooth) vector bundles is isomorphic to the category $DBun(SMan)$ of differential bundles in the category of smooth manifolds.
In \Cref{sub:vbun-are-dbun} we define a functor $\Psi:VBun \rightarrow DBun(SMan)$.
Then in \Cref{sub:dbun-are-vector} we define a functor $\Phi:DBun(SMan) \rightarrow VBun$ and show that $\Psi$ and $\Phi$ are inverses.

In this paper we follow Definition 5.9 in \cite{MR2723362} and work with manifolds that may have different dimensions in different connected components.
One advantage of this definition is that the category of smooth manifolds is idempotent complete.
The traditional definition of smooth manifold insists that all of the local co-ordinate systems have the same dimension even across different connected components (see for instance Definition 2.1 of \cite{MR1138207} and the \emph{pure manifolds} of 1.1 of \cite{MR1202431}).
The main result of this paper (differential bundles in the category of smooth manifolds are vector bundles) also holds for the traditional definition of a manifold and all of our proofs remain unchanged in this case.

\subsection{Vector bundles are differential bundles}
\label{sub:vbun-are-dbun}

In this section we define a functor $\Psi:VBun \rightarrow DBun(SMan)$ from the category of smooth vector bundles to the category of differential bundles in the category of smooth manifolds.
The main result of this paper is that $\Psi$ is invertible which we prove in \Cref{sub:dbun-are-vector}.
Since we allow our manifolds to have different dimensions in different connected components, it is natural to allow the dimension of the fibres of our vector bundles to have different dimensions in different connected components also.
Therefore the definition of a vector bundle that we use is a slight generalisation of the definition in Section 12.3 of \cite{MR2723362}.

\begin{definition}
  A \emph{(smooth) vector bundle} consists of a map $q:E \rightarrow M$ in the category of smooth manifolds such that:
  \begin{itemize}
    \item for all $m\in M$ each fibre $q^{-1}(m)$ is a vector space
    \item for all $m\in M$ there exists an open neighbourhood $U$ of $m$, a natural number $r$ and a fibre-preserving diffeomorphism $\phi:q^{-1}(U) \rightarrow U\times \mathbb{R}^r$ such that for all $x\in U$ the map $\phi|_{q^{-1}(x)}$ is a vector space isomorphism.
  \end{itemize}
\end{definition}

In the particular case where $E$ and $M$ have constant global dimensions $l$ and $k$ respectively the third condition implies that every fibre of a vector bundle $q: E \rightarrow M$ has dimension $l-k$.
(I.e. the dimension of the fibres is globally constant and we recover the definition of rank-$(l-k)$ vector bundle in Section 12.3 of \cite{MR2723362}.) 
Note that in the general case the dimension of the fibres of $q$ is still constant within each connected component of $M$ but is not necessarily globally constant.

\begin{proposition}\label{prop:vector-bundles-are-differential-bundles}
  Every vector bundle in the category of smooth manifolds is a differential bundle.
\end{proposition}
\begin{proof}
If $q:E \rightarrow M$ is a vector bundle and $m\in M$ then we can write an element of $E$ in local coordinates as $(m, a)$ and an element of $T(E)$ in local coordinates as $(m, a, v, b)$ where $m\in M$, $a, b\in \mathbb{R}^r$ and $v\in \mathbb{R}^k$ where $r$ is the dimension of $q^{-1}(m)$ and $k$ is the dimension of $M$ in the component containing $m$.
Then $+_q:E_2 \rightarrow E$ is given by $+_q(m, a, b) = (m, a+b)$, $\xi:M \rightarrow E$ by $\xi(m) = (m, 0)$ and $\lambda:E \rightarrow T(E)$ by $\lambda (m, a) = (m, 0, 0, a)$.

First, we check the axioms of a pre-differential bundle.
The equality $q\xi = id$ is immediate.
Next
\[\lambda\xi(m) = \lambda(m, 0) = (m, 0, 0, 0) = 0(m, 0) = 0\xi (m)\]
and 
\[
  T(\lambda)\lambda(m, a) = T(\lambda)(m, 0, 0, a)
  = (m, 0, 0, 0, 0, 0, 0, a)
  = \ell(m, 0, 0, a)
  = \ell\lambda(m, a)
\]
because $T(\lambda)(m, a, v, b) = (m, 0, 0, a, v, 0, 0, b)$.
Now we check the universality of the lift.
So suppose that $\vec{a}:A \rightarrow T(E)$ satisfies $T(q)\vec{a} = 0q p\vec{a}$ and $p\vec{a} = \xi q p\vec{a}$ as in the diagram
\[\begin{tikzcd}
A \arrow[bend left]{rrd}{\vec{a}} \arrow[bend right]{ddr}[swap]{qp\vec{a}} \drar[dashed]{f} & &\\
& E \rar{\lambda} \dar{q} & T(E) \dar{(T(q), p)}\\
& M \rar{(0, \xi)} & T(M) \times E
\end{tikzcd}\]
and suppose that $\vec{a} = (m, a, v, b)$ in local coordinates.
Then the condition $T(q)\vec{a} = 0q p\vec{a}$ implies that $v = 0$ and the condition $p\vec{a} = \xi q p\vec{a}$ implies that $a = 0$.
Therefore there exists a factorisation $f: A \rightarrow E$ given by in local co-ordinates by $(m, a)$ which is unique because $\lambda$ is a monomorphism.
\end{proof}

\Cref{prop:vector-bundles-are-differential-bundles} formulates the lift $\lambda$ in terms of local coordinates.
The following remark reformulates the lift in terms of the scalar multiplication.

\begin{remark}\label{rem:lift-in-terms-of-multiplication}
 The function $T(\bullet_q)$ is given in local co-ordinates by the following formula:
 \[
   (s,t) \bullet^T_q (x,a,v,b) = (x,s\bullet_qa, v, t\bullet_qa + s\bullet_q b)
 \]
 therefore
 \[
    \partial 0^{\mathbb{R}}! \bullet^T_q 0a  = (0,1) \bullet^T_q (x,a,0,0)  = (x,0,0,a) =
   \lambda a
 \]
 where $0^{\mathbb{R}}$ is the additive unit of $\mathbb R$ and $\partial: \mathbb R \to T(\mathbb R)$ is defined by
 $x \mapsto (x,1)$.
 Therefore $\lambda = \partial 0^{\mathbb R}! \bullet^T_q 0$.
\end{remark}

\begin{lemma}
 The function defined in \Cref{prop:vector-bundles-are-differential-bundles} extends to a functor $\Psi:VBun \rightarrow DBun(SMan)$ from the category of vector bundles to the category of differential bundles in the category of smooth manifolds.
 \end{lemma}
 \begin{proof}
 The action of $\Psi$ on objects is given in \Cref{prop:vector-bundles-are-differential-bundles}.
 The action of $\Psi$ on arrows is the identity function.
 For this to make sense we need to check that if
 \[\begin{tikzcd}
   E_0 \dar{q_0} \rar{f} & E_1 \dar{q_1}\\
   M_0 \rar{f_0} & M_1
 \end{tikzcd}\]
 is a morphism of vector bundles then $\lambda_1 f = T(f) \lambda_0$ where $\lambda_i$.
 We use the formulation of the lifts $\lambda_i$ in terms of the scalar multiplication given in \Cref{rem:lift-in-terms-of-multiplication}:
 \begin{align*}
  \lambda_1 f &= T(\bullet_1)(\partial 0^{\mathbb R}!, 0) f\\
  &= T(\bullet_1)(\partial 0^{\mathbb R}!, 0 f)\\
  &= T(\bullet_1)(\partial 0^{\mathbb R}!, T(f)0)\\
  &= T(\bullet_1)(id \times T(f))(\partial 0^{\mathbb R}!, 0)\\
  &= T(\bullet_1 (id \times f))(\partial 0^{\mathbb R}!, 0)\\
  &= T(f \bullet_0)(\partial 0^{\mathbb R}!, 0) = T(f)\lambda_0
 \end{align*}
 where the penultimate equality follows from the fact that $f$ preserves scalar multiplication.
\end{proof}

\subsection{Differential bundles in smooth manifolds are vector bundles}
\label{sub:dbun-are-vector}

In \Cref{sub:vbun-are-dbun} we constructed a functor $\Psi:VBun \rightarrow DBun(SMan)$ from the category of smooth vector bundles to the category of differential bundles in the category of smooth manifolds.
In this section we construct a functor $\Phi:DBun(SMan) \rightarrow VBun$ and show that $\Psi$ and $\Phi$ are inverses.
Our general strategy is to recall that the category of vector bundles is closed under pullback and idempotent splittings.
Then we can apply \Cref{cor:splitting-of-differential-bundles} to obtain our result.
So first we recall that the pullback of a vector bundle is a vector bundle.

\begin{lemma}\label{lem:pullback-of-vector-bundle}
    Let $q_0: E_0 \rightarrow M_0$ be a vector bundle with addition $+_0$, zero $\xi_0$ and scalar multiplication $\bullet_0$.
    If
    \[\begin{tikzcd}
    E_1 \rar{\iota_0} \dar{q_1} & E_0 \dar{q_0} \\
    M_1 \rar{f} & M_0
    \end{tikzcd}\]
    is a pullback then $q_1:E_1 \rightarrow M_1$ is a vector bundle with addition $+_1$, zero $\xi_1$ and multiplication $\bullet_1$ such that:
    \begin{itemize}
        \item $\iota_0(a+_1 b) = \iota_0 a +_0 \iota_0 b$
        \item $\iota_0 \xi_1 = \xi_0 f$
        \item $\iota_0(r\bullet_1 a) = r\bullet_0 \iota_0 a$
    \end{itemize}
    where $a, b \in E_1$ such that $q_1(a) = q_1(b)$.
\end{lemma}

Next we recall that the class of vector bundles is Cauchy complete (every idempotent splits).

\begin{lemma}\label{lem:image-idempotent-vector-bundle}
  Let $q_1:E_1 \rightarrow M_1$ be a vector bundle with addition $+_1$, zero $\xi_1$ and multiplication $\bullet_1$.
  If
  \[\begin{tikzcd}
  E_1 \rar{\phi} \dar{q_1} & E_1 \dar{q_1}\\
  M_1 \ar[equals]{r} & M_1
  \end{tikzcd}\]
  is an idempotent vector bundle endomorphism over the fixed base space $M$ then its image $q_2:E_2 \rightarrow M_1$ is a vector bundle with addition $+_2$, zero $\xi_2$ and multiplication $\bullet_2$ such that:
  \begin{itemize}
      \item $S(a+_2 b) = Sa +_1 Sb$
      \item $S\xi_2 = \xi_1$
      \item $S(r\bullet_2 a) = r\bullet_1 Sa$
  \end{itemize}{}
  where $S$ is the equaliser of $\phi$ and $id$ and where $a, b \in E_2$ such that $q_2(a) = q_2(b)$.
\end{lemma}
\begin{proof}
    Since the morphism between base spaces is the identity, we only need to consider the image of $\phi$.
    Following Proposition 1 in \cite{MR143225} it suffices to check that $\phi_1$ has locally constant rank. However since $\phi_1$ is a projection its rank is equal to its trace.
    Since the trace is continuous and the rank takes integer values, we conclude that the rank of $\phi_1$ is locally constant.
\end{proof}

\begin{remark}
    The proof of  \cref{lem:image-idempotent-vector-bundle} remains unchanged for the traditional definition of a smooth manifold because the map between the base spaces is the identity.
\end{remark}

The following example shows that the above result does not necessarily hold if
$(\phi, id)$ is not idempotent.

\begin{example}
  The kernel of the vector bundle morphism
  \[
    (\phi, id): (\mathbb{R}\times \mathbb{R}, \pi_0) \rightarrow
    (\mathbb{R}\times \mathbb{R}, \pi_0)
  \]
  where $\phi: (x, r) = (x, r\bullet x)$ is the union of $\{0\}\times
  \mathbb{R}$ and $(\mathbb{R}\setminus \{0\})\times \{0\}$ in the category of
  topological bundles. Therefore the kernel does not necessarily exist in the
  category of smooth vector bundles if $\phi_1$ is not idempotent.
\end{example}

Now we combine our previous results to prove the main result of this paper.

\begin{proposition}\label{prop:differential-bundles-are-vector-bundles}
    If $(q:E \rightarrow M, +, \xi, \lambda)$ is a differential bundle in the category of smooth manifolds then $q$ is the projection of a vector bundle with addition $+$, zero $\xi$ and scalar multiplication $\bullet$ satisfying $\lambda (r\bullet a) = r\bullet_p \lambda a$.
\end{proposition}
\begin{proof}
    \Cref{cor:splitting-of-differential-bundles} shows that $q$ is a retract of a pullback of $p:T(E) \rightarrow E$.
    First apply \Cref{lem:pullback-of-vector-bundle} with $\iota_0 = \iota_M$ and second apply \Cref{lem:image-idempotent-vector-bundle} with $S = (\lambda, q)$.
    Since $\lambda = \iota_M (\lambda, q)$ therefore $q: E \rightarrow M$ is a vector bundle with addition $+_2$, zero $\xi_2$ and scalar multiplication $\bullet_2$ such that:
    \begin{itemize}
        \item $\lambda(a +_2 b) = \lambda a +_p \lambda b$,
        \item $\lambda \xi_2 = 0 \xi$
        \item $\lambda(r \bullet_2 a) = r\bullet_p \lambda a$.
    \end{itemize}{}
    Therefore $+_2 = +$ because $\lambda:(E, q)\Rightarrow (T(E), p)$ is an additive bundle morphism.
    Also $\xi_2 = \xi$ because of the pre-differential bundle axiom $\lambda \xi = 0 \xi$.
\end{proof}

\begin{lemma}\label{lem:functor-Phi}
The function defined in \Cref{prop:differential-bundles-are-vector-bundles} extends to a functor $\Phi:DBun(SMan) \rightarrow VBun$ from the category of differential bundles in the category of smooth manifolds to the category of vector bundles.
\end{lemma}
\begin{proof}
The action of $\Phi$ on objects is defined in \Cref{prop:differential-bundles-are-vector-bundles}.
Note that the scalar multiplication is the unique one satisfying $\lambda (r\bullet_q a) = r\bullet_p \lambda a$ because $\lambda$ is a monomorphism.
The action of $\Phi$ on arrows is the identity function.
For this to make sense we need to check that if $f$ and $f_0$ satisfy
\[\begin{tikzcd}
 E_0 \rar{f} \dar{q_0} & E_1 \dar{q_1} \\
 M_0 \rar{f_0} & M_1
\end{tikzcd}\]
and $\lambda f = T(f)\lambda$ then $f$ preserves the addition and scalar multiplication of $\Phi(q_0)$.
To see that the addition is preserved we refer to Proposition 2.16 of \cite{MR3684725}.
To see that the scalar multiplication is preserved we calculate:
\[
 \lambda f (r \bullet_0 a) = T(f)\lambda (r \bullet_0 a)
 = T(f)(r\bullet_p \lambda a)
 = r \bullet_p T(f)\lambda a
 = r\bullet_p \lambda fa
 = \lambda (r \bullet_1 fa)
\]
and so $f(r\bullet_0 a) = r \bullet_1 fa$ because $\lambda$ is a monomorphism.
\end{proof}

\isoOfCategories
\begin{proof}
We show that the functors $\Psi$ and $\Phi$ (defined in \Cref{prop:vector-bundles-are-differential-bundles} and \Cref{lem:functor-Phi} respectively) are inverses.
Since the action of both $\Phi$ and $\Psi$ on arrows is the identity function we only need to consider the action on objects.
Since $\Phi$ and $\Psi$ leave the projection, addition and zero section unchanged it in fact only remains to consider the lift and scalar multiplication.

In one direction let $(q:E \rightarrow M, \xi, +, \lambda)$ be a differential bundle in the category of smooth manifolds.
\Cref{prop:differential-bundles-are-vector-bundles} implies that $\Phi(q)$ has scalar multiplication satisfying $\lambda \bullet = \bullet_p (id\times \lambda)$.
\Cref{rem:lift-in-terms-of-multiplication} implies that $\Psi\Phi(q)$ has lift $\lambda_2 = T(\bullet)(\partial 0^{\mathbb R}!, 0)$.
We need to show that $\lambda = \lambda_2$:
\[
 T(\lambda)\lambda_2 = T(\lambda)T(\bullet)(\partial 0^{\mathbb R}!, 0)
 = T(\bullet_p)T(id\times \lambda)(\partial 0^{\mathbb R}!, 0)
 = T(\bullet_p)(\partial 0^{\mathbb R}!, T(\lambda)0)
 = T(\bullet_p)(\partial 0^{\mathbb R}!, 0)\lambda = \ell\lambda = T(\lambda)\lambda 
\]
and so $\lambda_2 = \lambda$ because $T(\lambda)$ is a monomorphism.

In the other direction let $(q:E \rightarrow M, \xi, +, \bullet)$ be a vector bundle.
\Cref{rem:lift-in-terms-of-multiplication} implies that $\Psi(q)$ is a differential bundle with lift given by $\lambda(m, a) = (m, 0, 0, a)$ in local coordinates.
\Cref{prop:differential-bundles-are-vector-bundles} implies that $\Phi\Psi(q)$ is a vector bundle with scalar multiplication $\bullet_2$ satisfying $\lambda\bullet_2 = \bullet_p(id\times \lambda)$.
We need to show that $\bullet = \bullet_2$:
\[ \lambda(r \bullet_2 (m, a)) = r\bullet_p \lambda(m, a)
 = r\bullet_p (m, 0, 0, a)
 = (m, 0, 0, r\cdot a)
 = \lambda(m, r\cdot a)
 =\lambda (r \bullet (m, a))
\]
and so $\bullet = \bullet_2$ because $\lambda$ is a monomorphism.
\end{proof}

We conclude by showing that it was in fact necessary to use the universality of the lift in the proof of \Cref{prop:differential-bundles-are-vector-bundles}.

\begin{example}\label{ex:prediff}
  We give an example of a pre-differential bundle (as defined in \cref{def:pre-differential-bundle}) in the category of smooth manifolds that is not a vector bundle.
  Let $q:\mathbb{R}^2 \rightarrow \mathbb{R}$ be defined by $q(x, y) = (1-\delta(y))x+\delta(y)x^3$ where $\delta$ is a smooth and monotonic bump function that is $0$ for $y\leq 0$ and $1$ for $y \geq 1$.
  To see that $q$ cannot be the projection of a vector bundle recall that every vector bundle is, in particular, a submersion.
  However, the derivative of $q$ the point $(0, 1)$ vanishes so $q$ is not a submersion.
  
  Now we show that $q$ is a pre-differential bundle with lift $\lambda(x, y) = (q(x, y), 0, 0, y)$ and zero section $\xi(z) = (z, 0)$. First we check:
  \begin{itemize}
  \item $q\xi(z) = q(z, 0) = (1-\delta(0))z+\delta(0)z^3 = z$
  \item $\lambda\xi(z) = \lambda(z, 0) = (z, 0, 0, 0) = 0\xi(z)$
  \end{itemize}
  and to check that $T(\lambda)\lambda = \ell\lambda$ we first note that in general:
  \[
      T(\lambda)(x, y, v, w) = (q(x, y), 0, 0, y, (1-\delta(y)+\delta(y)3x^2)v+(x^2-1)\delta'(y)xw, 0, 0,w)
  \]
  and so
  \begin{align*}
    T(\lambda)\lambda(x, y) &= T(\lambda)(q(x, y), 0, 0, y)\\
    &= (q(x, y), 0, 0, 0, 0, 0, 0, y)\\
                               &= \ell(q(x, y), 0, 0, y) = \ell\lambda(x, y)
  \end{align*}
  as required.
  \end{example}

\section{Acknowledgements}
This work would not have been possible without the contributions of Matthew Burke, who provided several key insights in the earlier stages of the paper.
The author would also like to thank Robin Cockett for his help in the editing process.

\subsection{Author statements}
\paragraph{Data availability statement}
Data sharing is not applicable to this article as no new data were created or analyzed in this study. 

\paragraph{Competing financial interests statement}
Benjamin MacAdam declares they have no competing financial interests.
\bibliography{references}{} \bibliographystyle{abbrv}
\end{document}